\documentclass[a4paper,11pt]{amsart}
\usepackage{amsmath,amsthm,amssymb,amsfonts,enumerate,color,esint}
\usepackage[pdftex]{graphicx}
\usepackage{float}

\usepackage[colorlinks=true, allcolors=blue]{hyperref}
\usepackage{amsrefs}

\usepackage[ansinew]{inputenc}
\usepackage[dvips]{epsfig}
\usepackage{graphicx}
\usepackage[english]{babel}
\usepackage{tikz}
\usepackage{hyperref}

\theoremstyle{plain}
\newtheorem{thm}{Theorem}[section]

\theoremstyle{definition}

\newtheorem{rem}[thm]{Remark}

\newcommand{\R}{\mathbb{R}}

\newcommand{\M}{\mathcal{M}}

\newcommand{\A}{\mathcal{A}}
\newcommand{\J}{\mathcal{J}}
\newcommand{\I}{\mathcal{I}}

\newcommand{\dive}{\operatorname{div}}

\newcommand{\BV}{\operatorname{BV}}
\newcommand{\vep}{\varepsilon}
\newcommand{\Lip}{\operatorname{Lip}}

\numberwithin{equation}{section}

\allowdisplaybreaks

\author[L. A. Caffarelli]{Luis A. Caffarelli}
\address{Department of Mathematics\\
The University of Texas at Austin\\
2515 Speedway, Austin\\
TX 78712, United States of America}
\email{caffarel@math.utexas.edu}

\author[P. R. Stinga]{Pablo Ra\'ul Stinga}
\address{Department of Mathematics\\
Iowa State University\\
396 Carver Hall, Ames\\
IA 50011, United States of America}
\email{stinga@iastate.edu}

\author[H. Vivas]{Hern\'an Vivas}
\address{Centro Marplatense de Investigaciones Matem\'aticas/CONICET\\
Dean Funes 3350\\
7600 Mar del Plata, Argentina}
\email{havivas@mdp.edu.ar}

\thanks{Research partially supported by NSF grant 1500871 (USA), Simons Foundation grant 580911 (USA), and
Agencia Nacional de Promoci\'on Cient\'ifica y Tecnol\'ogica grant PICT 2019-3530 (Argentina).}



\begin{document}

\title[Surfaces of minimum mean curvature variation]{A PDE approach to the existence and regularity
of surfaces of minimum mean curvature variation}

\begin{abstract}
We develop an analytic theory of existence and regularity of surfaces 
(given by graphs) arising from the geometric minimization problem
\[
\min_\M\frac{1}{2}\int_\M|\nabla_{\M}H|^2\,dA
\]
where $\M$ ranges over all $n$-dimensional manifolds in $\R^{n+1}$ with prescribed boundary, $\nabla_{\M}H$ is the 
tangential gradient along $\M$ of the mean curvature $H$ of $\M$ and $dA$ is the differential of surface area.
The minimizers, called surfaces of minimum mean curvature variation,
are central in applications of computer-aided design, computer-aided manufacturing and mechanics.
Our main results show the existence of both smooth surfaces and
of variational solutions to the minimization problem together with geometric regularity results in the case of graphs.
These are the first analytic results available on the literature for this problem.
\end{abstract}

\maketitle

\section{Introduction}

Our aim is to develop an analytic theory of existence and regularity of surfaces of minimum mean curvature variation,
that is, surfaces arising from the geometric minimization problem
\begin{equation}\label{eq.min}
\min_\M\frac{1}{2}\int_\M|\nabla_{\M}H|^2\,dA
\end{equation}
in the case where the surface $\M$ is the graph of a function $u$. Here $\M$ ranges over all $n$-dimensional manifolds in $\R^{n+1}$, $n\geq1$, with prescribed boundary, $H$ is the mean curvature of $\M$, $\nabla_{\M}$ represents the tangential gradient over $\M$ and $dA$ is the differential of surface area. Since \eqref{eq.min} minimizes the quadratic variation of the mean curvature of $\M$, surfaces with
constant mean curvature such as spheres, cylinders, planes and minimal surfaces have zero energy.

Our main results are the existence of smooth surfaces of minimum mean curvature variation,
see Theorems \ref{thm:simplifiedexistence} and \ref{thm:geometricexistence}, and the existence
and regularity of variational solutions to \eqref{eq.min}, see Theorem \ref{thm.weak}.

There are several difficulties when considering the geometric minimization problem \eqref{eq.min}.
One main obstacle is the highly nonlinear and degenerate nature of the problem.
If a manifold $\M$ is a smooth minimizer then, by performing normal variations, it can be seen that the Euler--Lagrange equation
satisfied by the mean curvature $H$ of $\M$ is
$$\Delta^2_\M H+2(2H^2-K)\Delta_\M H+2\langle\nabla_\M H,\diamond H\rangle-2H|H|^2=0$$
where $\Delta_\M$ is the Laplace--Beltrami operator on $\M$, $K$ is the Gaussian curvature  and
$\diamond H$ denotes second tangent operator over $\M$, see \cite{Xu-Zhang}. This is a nonlinear sixth-order equation
for $u$ (the function that represents the graph of $\M$) which, at this moment, seems to be analytically quite intractable. In fact, comparison principles
and uniqueness of solutions for this equation are not known. If one writes $H$ in terms of
a parametrization of $\M$ and looks at \eqref{eq.min} as an energy
given by third order derivatives of the parametrization then the nonlinear degenerate structure
prevents from applying minimization techniques in typical Hilbert spaces like the Sobolev space $W^{3,2}(\Omega)$.
In addition, the solvability of the prescribed mean curvature equation depends in a nontrivial way on geometric properties
of the boundary, see \cite{Gilbarg-Trudinger}.

One of the motivations for studying \eqref{eq.min} comes from computer-aided design (CAD) and computer-aided manufacturing (CAM)
problems in engineering, aerospace industry, computer animation and architectural design.
A typical problem in CAD and CAM is the robust design of fair surfaces and
the creation of shapes such that some aspects of the design process
are optimized, see \cite{Sacar-Rao-Narayan}. In many instances, the goal is the creation of complex,
smoothly shaped models and surfaces with specified geometric constraints, see \cite{Welch-Witkin}.
Typically, these problems are approached via a variational principle. In this context,
in 1992, Moreton and S\'equin proposed in \cite{Moreton-Sequin} a numerical algorithm
to create $2$-dimensional fair surfaces $\M$ as minimizers of the geometric energy functional in \eqref{eq.min}.
In \cite{Xu-Zhang}, a finite difference method was proposed to construct surfaces
as steady states of a sixth order flow derived from the Euler--Lagrange equation associated with \eqref{eq.min}.

Perhaps the most important aspect in the applications mentioned above is to obtain surfaces that preserve several degrees of geometric continuity
where two different surfaces meet.
In particular, global continuity of the mean curvature is fundamental in
concrete problems such as the design of streamlined surfaces of aircrafts,
ships and cars, and numerical evidence of this regularity has been observed in \cite{Xu-Zhang}.
Furthermore, minimization of geometric functionals given in terms of the mean curvature are important
in continuum mechanics as they account for the bending energy of elastic materials \cite{Capovilla,Zorin}. 

Up to the best of our knowledge, the analytical theory for the existence
of surfaces of minimum mean curvature variation in general is missing.
No proof of regularity of minimizers and their mean curvature has been available thus far. 
Our aim is to fill these gaps and establish an analytical foundation from the PDE
and variational perspectives.

Before describing our results, we set up the notation.
From now on, we assume that $\M$ is the graph of a real-valued
function $u$ defined on a bounded domain $\Omega\subset\R^n$, $n\geq1$,
namely, $\M=\left\{(x,u(x)):x\in\overline{\Omega}\right\}$.
For $x\in \Omega$,
the upward pointing unit normal at $(x,u(x))$ is
\[
\nu(x)=\frac{(-\nabla u(x),1)}{(1+|\nabla u(x)|^2)^{1/2}}
\]
and the mean curvature $H=H(x)$ is
\[
H=\frac{1}{n}\dive\left(\frac{\nabla u}{(1+|\nabla u|^2)^{1/2}}\right).
\]
We let $D(u):=(1+|\nabla u|^2)^{1/2}$ so that $dA=D(u)\,dx$. It is not difficult to see that
$$|\nabla_\M H|^2=|\nabla H|^2-\bigg|\frac{\nabla u\cdot\nabla H}{D(u)}\bigg|^2$$
where $\nabla$ is the usual Euclidean gradient. Then the geometric energy in \eqref{eq.min} becomes
\begin{equation}\label{eq.geometricenergy}
E[\M]:=\frac{1}{2}\int_\Omega\bigg[|\nabla H|^2-\bigg|\frac{\nabla u\cdot\nabla H}{D(u)}\bigg|^2\bigg]D(u)\,dx.
\end{equation}
By the Cauchy--Schwartz inequality,
$$\frac{|\nabla H|^2}{D(u)^2}\leq|\nabla_\M H|^2\leq|\nabla H|^2.$$
Therefore, we will also consider the (larger) simplified energy functional
\begin{equation}\label{eq.energy}
E[H,u]:=\frac{1}{2}\int_{\Omega}|\nabla H|^2D(u)\,dx.
\end{equation}

In Section \ref{section:simplifiedexistence} we prove the existence of a smooth function $u$
that satisfies the prescribed mean curvature equation and whose mean curvature $H$
is regular up to the boundary $\partial\Omega$ and, in addition, minimizes the simplified energy functional
\eqref{eq.energy}, see Theorem \ref{thm:simplifiedexistence}. Then, in Section \ref{section:geometricexistence},
we show how to modify our procedure to construct regular solutions for the case of
the geometric energy functional \eqref{eq.geometricenergy}, see Theorem \ref{thm:geometricexistence}.
We believe that our method of proof, which is based on a fixed-point argument, will help establish
numerical and computational schemes for the construction of minimizers whose mean curvature
is guaranteed to be globally continuous. Indeed, our iterative procedure
begins by linearizing the energy at some $v$, so that the
numerical iteration may be initiated at, say, a minimal surface $v$,
and continued by solving a linear Euler--Lagrange equation for $H$.
Moreover, our regularity results are sharp and global, not requiring any extra additional assumptions on
the boundary of the domain, so they will be useful for proving precise rates of convergence.

Next, in Section \ref{section:weakexistence}, we establish a variational formulation
for constructing minimizers of \eqref{eq.energy} and show global regularity results for $u$ and $H$, see Theorem \ref{thm.weak}.
We highlight that setting up an appropriate functional setting to look for minimizers is a nontrivial task for
reasons like the ones we mentioned before. For instance,
the Euler--Lagrange equation seems quite intractable with the available tools and
the nonlinear degenerate structure prevents us from using typical Hilbert space methods.
We overcome these difficulties by looking for a minimizer pair $(u,H)$,
where $u$ is the solution to the variational or weak formulation of the prescribed mean curvature equation
with right hand side $H$. Our functional setting requires mild conditions on $u$ and $H$ on which
we can use geometric measure theory tools to recover enough compactness. We also show that
our minimizer pairs $(u,H)$ satisfy the required mean curvature global continuity required in applications \cite{Moreton-Sequin,Xu-Zhang}.

\section{Existence and regularity for the simplified energy}\label{section:simplifiedexistence}

In this section we work with the simplified energy functional \eqref{eq.energy}.
Consider a bounded domain $\Omega$ such that $\partial\Omega\in C^{2,\alpha}$, for some $0<\alpha<1$ fixed.
We assume that we are given prescribed boundary values $g\in C^{2,\alpha}(\overline{\Omega})$
for $u$ and $h\in C^{1,\alpha}(\overline{\Omega})$ for $H$ on $\partial\Omega$.

We address the following problem: given $\Omega$ and the boundary datum $g$, find a surface
$\M$ given by the graph of a function $u$
such that $u=g$ on $\partial\Omega$ and its mean curvature $H$ is a minimizer of \eqref{eq.energy} among all
functions with prescribed boundary values $h$. 

Consider the Banach space $\mathfrak{B}=C^{1,\alpha}(\overline{\Omega})$ and its subset
$\mathfrak{G}:=\big\{v\in C^{1,\alpha}(\overline{\Omega}):v=g~\hbox{on}~\partial\Omega\big\}$.
Clearly, $\mathfrak{G}$ is nonempty, closed and convex.
For any $v\in\mathfrak{G}$, we define the functional
\begin{equation}\label{eq.Ev}
E[H,v]:=\frac{1}{2}\int_{\Omega}|\nabla H|^2D(v)\,dx.
\end{equation}

A map $T:\mathfrak{G}\to\mathfrak{G}$ is constructed in a two-step process.
First, given any $v\in\mathfrak{G}$, let $H\in W^{1,2}(\Omega)$ be the unique minimizer
to \eqref{eq.Ev} such that $H-h\in W^{1,2}_0(\Omega)$, which exists because 
$1\leq D(v)\leq(1+\|\nabla v\|_{L^\infty(\Omega)}^2)^{1/2}$. Then $H$ is the unique weak solution to
$$\begin{cases}
\dive(D(v)\nabla H)=0&\hbox{in}~\Omega\\
H=h&\hbox{on}~\partial\Omega.
\end{cases}$$
Since $v\in C^{1,\alpha}(\overline{\Omega})$, by global Schauder estimates (see \cite{Gilbarg-Trudinger}),
\begin{equation}\label{eq:SchauderH}
\|H\|_{C^{1,\alpha}(\overline{\Omega})}\leq C_n[\partial\Omega]_{C^{1,\alpha}}\|D(v)\|_{C^{0,\alpha}(\overline{\Omega})}
\|h\|_{C^{1,\alpha}(\partial\Omega)}
\end{equation}
where $C_n>0$ is a constant that depends only on $n$.

Second, having constructed this $H$, we find the solution $u$ to the prescribed
mean curvature equation
\begin{equation}\label{eq:pmc}
\begin{cases}
\displaystyle\dive\bigg(\frac{\nabla u}{D(u)}\bigg)=nH&\hbox{in}~\Omega\\
u=g&\hbox{on}~\partial\Omega.
\end{cases}
\end{equation}
For this, we use the following sharp existence result, see \cite{Gilbarg-Trudinger}.

\begin{thm}\label{thm.pmc}
Let $\Omega\subset\R^n$ be a bounded domain with $C^{2,\alpha}$ boundary, for some $0<\alpha<1$.
Suppose that $H\in C^1(\overline{\Omega})$ satisfies
\begin{equation}\label{eq.conLnH}
\|H\|_{L^n(\Omega)}<|B_1|^{1/n}
\end{equation}
and, for any $y\in\partial\Omega$,
\begin{equation}\label{eq.mcbdry}
|H(y)|\leq \frac{n-1}{n}H_{\partial\Omega}(y)
\end{equation}
where $|B_1|$ is the measure of the unit ball in $\R^n$ and
$H_{\partial\Omega}$ is the mean curvature of $\partial\Omega$ corresponding to the inner unit normal vector to $\partial\Omega$. 
Then for any $g\in C^{2,\alpha}(\overline{\Omega})$ there exists a unique solution $u\in C^{2,\alpha}(\overline{\Omega})$ 
to \eqref{eq:pmc}. In particular, there exists a constant $C_\ast>0$ depending only on $n$, $\alpha$, $\|H\|_{L^n(\Omega)}$,
$\|H\|_{C^1(\overline{\Omega})}$, $\|g\|_{C^{2,\alpha}(\overline{\Omega})}$ and $\Omega$ such that
\begin{equation}\label{eq:C2alphabound}
\|u\|_{C^{2,\alpha}(\overline{\Omega})}\leq C_\ast.
\end{equation}
\end{thm}

Now \eqref{eq.conLnH} and \eqref{eq.mcbdry} impose further restrictions on the boundary values $h$ of $H$
(see also Remark \ref{rem:nonexistence}). By the maximum principle, if we assume that
\begin{equation}\label{eq:h2}
\max_{\partial\Omega}|h|<\left(\frac{|B_1|}{|\Omega|}\right)^{1/n}
\end{equation}
then \eqref{eq.conLnH} holds. Condition \eqref{eq.mcbdry} is natural
and cannot be avoided (see Remark \ref{rem:nonexistence}). Therefore, we assume that
$h$ additionally satisfies
\begin{equation}\label{eq:h1}
|h(y)|\leq \frac{n-1}{n}H_{\partial\Omega}(y)\qquad\hbox{for all}~y\in\partial\Omega.
\end{equation}
Therefore, under the additional assumptions \eqref{eq:h2} and \eqref{eq:h1}, we can apply Theorem \ref{thm.pmc} and
find the unique solution $u\in C^{2,\alpha}(\overline{\Omega})$ to \eqref{eq:pmc}. This completes the second step.

We then define $T:\mathfrak{G}\to\mathfrak{G}$ by $T(v)=u$.

Let us prove that $T$ is continuous. Fix $v_1\in\mathfrak{G}$ and $\varepsilon>0$.
We need to show that there exists $\delta=\delta(\vep,v_1)>0$ such that
for any $v_2\in\mathfrak{G}$ satisfying $\|v_1-v_2\|_{C^{1,\alpha}(\overline{\Omega})}<\delta$ we have
$\|u_1-u_2\|_{C^{1,\alpha}(\overline{\Omega})}<\vep$, where $u_j=Tv_j$, for $j=1,2$.
Let $H_j$ denote the minimizer of $E[\cdot,v_j]$, $j=1,2$, as constructed in the first step. Then
$H=H_1-H_2$ is the unique weak solution to
$$\begin{cases}
\dive(D(v_1)\nabla H)=\dive\big((D(v_2)-D(v_1))\nabla H_2\big)&\hbox{in}~\Omega\\
H=0&\hbox{on}~\partial\Omega.
\end{cases}$$
By global Schauder estimates,
\begin{equation}\label{eq:primera}
\begin{aligned}
\|H\|_{C^{1,\alpha}(\overline{\Omega})} &\leq C_n[\partial\Omega]_{C^{1,\alpha}}\|D(v_1)\|_{C^{0,\alpha}(\overline{\Omega})}
\|(D(v_2)-D(v_1))\nabla H_2\|_{C^{0,\alpha}(\overline{\Omega})} \\
&\leq C(n,\alpha,\Omega,v_1,\nabla H_2)\|v_1-v_2\|_{C^{1,\alpha}(\overline{\Omega})}=:C_1\|v_1-v_2\|_{C^{1,\alpha}(\overline{\Omega})}.
\end{aligned}
\end{equation}
To estimate the difference $u=u_1-u_2\in C^{2,\alpha}(\overline{\Omega})$, observe that
$$\begin{cases}
\displaystyle\dive\bigg(\frac{\nabla u_1}{D(u_1)}-\frac{\nabla u_2}{D(u_2)}\bigg)=nH &\hbox{in}~\Omega\\
u=0&\hbox{on}~\partial\Omega.
\end{cases}$$
The vector field
\[
F(p):=\frac{p}{\sqrt{1+|p|^2}}\qquad p\in\R^n
\]
is smooth and bounded and its gradient $\nabla F(p)$ is a bounded, symmetric matrix.
Moreover, $\nabla F(p)$ is locally strictly elliptic, namely, for any $\xi\in\R^n$,
$$\sum_{i,j=1}^n\partial_jF_i(p)\xi_i\xi_j \geq\frac{|\xi|^2}{D(p)^3}\geq\theta(R)|\xi|^2$$
for all $|p|<R$, where $\theta(R)\to 0$ as $R\to\infty$.
Furthermore,
\begin{align*}
F(\nabla u_1)-F(\nabla u_2) & = \int_0^1\frac{d}{dt}F(t\nabla u_1+(1-t)\nabla u_2)\,dt \\
& = \int_0^1\nabla F(t\nabla u_1+(1-t)\nabla u_2)\nabla (u_1-u_2)\,dt=:A(x)\nabla u
\end{align*}
where $A(x)$ is symmetric and bounded.
Now, $u_1\in C^{2,\alpha}(\overline{\Omega})$ is fixed. By \eqref{eq:C2alphabound},
the $C^{2,\alpha}(\overline{\Omega})$ norm of $u_2$
is uniformly controlled by the $C^1(\overline{\Omega})$ norm of $H_2$, which in turn is uniformly close to
the $C^1(\overline{\Omega})$ norm of the initially fixed $H_1$. Therefore,
$A(x)$ is strictly elliptic. Moreover, since $\nabla F(p)$ and $D^2F(p)$ are bounded
and $\nabla u_1,\nabla u_2\in C^{0,\alpha}(\overline{\Omega})$, it can be verified that
$$\|A\|_{C^{0,\alpha}(\overline{\Omega})}\leq M$$ 
with $M>0$ a constant depending only on $n$, $\alpha$, $\|H_1\|_{L^n(\Omega)}$,
$\|H_1\|_{C^1(\overline{\Omega})}$, $\|g\|_{C^{2,\alpha}(\overline{\Omega})}$, and $\Omega$, see \eqref{eq:C2alphabound}.
All of these quantities are independent of $u_2$ if $v_2$ is close to $v_1$ in $C^{1,\alpha}(\overline{\Omega})$.
In summary, we have found that $u$ is a solution to
$$\begin{cases}
\displaystyle\dive(A(x)\nabla u)=nH &\hbox{in}~\Omega\\
u=0&\hbox{on}~\partial\Omega
\end{cases}$$
and so, by Schauder estimates, 
\begin{equation}\label{eq:segunda}
\|u\|_{C^{1,\alpha}(\overline{\Omega})} \leq C_n[\partial\Omega]_{C^{1,\alpha}}M\|H\|_{C^{1,\alpha}(\overline{\Omega})}
=:C_2\|H\|_{C^{1,\alpha}(\overline{\Omega})}.
\end{equation}
Collecting estimates \eqref{eq:primera} and \eqref{eq:segunda} and recalling that
$u=u_1-u_2=Tv_1-Tv_2$, and $H=H_1-H_2$, we obtain $\|Tv_1-Tv_2\|_{C^{1,\alpha}(\overline{\Omega})}\leq C_1C_2\|v_1-v_2\|_{C^{1,\alpha}(\overline{\Omega})}$.
If we choose $\delta=\vep/(C_1C_2)$ then we conclude that $T$ is continuous, as desired.

Let us next prove that $T(\mathfrak{G})$ is precompact.
Let $\{v_k\}_{k\geq1}$ be a sequence in $\mathfrak{G}$ such that
$$\sup_{k\geq1}\|v_k\|_{C^{1,\alpha}(\overline{\Omega})}\leq N_1<\infty.$$
Consider the corresponding solutions $H_k\in C^{1,\alpha}(\overline{\Omega})$ found in the first step.
Set $u_k=Tv_k$. By \eqref{eq:C2alphabound},
$$\|u_k\|_{C^{2,\alpha}(\overline{\Omega})}\leq C_{k}$$
where $C_k>0$ is a constant depending only on $n$, $\alpha$, $\|H_k\|_{L^n(\Omega)}$, $\|H_k\|_{C^1(\overline{\Omega})}$, $\|h\|_{C^{2,\alpha}(\overline{\Omega})}$, and $\Omega$. Since all $H_k$ have the same boundary values $h$, by the maximum principle,
$$\sup_{k\geq1}\|H_k\|_{L^n(\Omega)}=N_2<\infty.$$ Furthermore, from the $C^{1,\alpha}$ estimate in \eqref{eq:SchauderH},
$$\sup_{k\geq1}\|H_k\|_{C^1(\overline{\Omega})}\leq C_n[\partial\Omega]_{C^{1,\alpha}}
\|h\|_{C^{1,\alpha}(\partial\Omega)}\sup_{k\geq1}\|D(v_k)\|_{C^{0,\alpha}(\overline{\Omega})}=N_3<\infty.$$
Consequently,
$$\sup_{k\geq1}\|u_k\|_{C^{2,\alpha}(\overline{\Omega})}\leq \sup_{k\geq1}C_{k}=N_4<\infty.$$
By the Arzel\`a--Ascoli compact embedding theorem, there exist a subsequence $\{u_{k_j}\}_{j\geq1}$ of $\{u_k\}_{k\geq1}$ and $u\in\mathfrak{G}$
such that $u_{k_j}\to u$ in $C^{1,\alpha}(\overline{\Omega})$, as desired.

Thus, by Schauder's fixed point theorem, there exists $u\in\mathfrak{G}$ such that $Tu=u$.
We have proved the following:

\begin{thm}[Existence for the simplified energy and regularity]\label{thm:simplifiedexistence}
Let $\Omega\subset\R^n$ be a bounded domain with $C^{2,\alpha}$ boundary $\partial\Omega$, for $0<\alpha<1$.
Fix $g\in C^{2,\alpha}(\overline{\Omega})$. Let $h\in C^{1,\alpha}(\overline{\Omega})$ such that
\begin{equation}\label{eq.condh2}
\max_{\partial\Omega}|h|<\left(\frac{|B_1|}{|\Omega|}\right)^{1/n}
\end{equation}
and
\begin{equation}\label{eq.condh1}
|h(y)|\leq\frac{n-1}{n}H_{\partial\Omega}(y)\qquad\hbox{for all}~y\in\partial\Omega,
\end{equation}
where $H_{\partial\Omega}$ is the mean curvature of $\partial\Omega$ corresponding to the inner unit normal vector to $\partial\Omega$.
Then there exist $u\in C^{2,\alpha}(\overline{\Omega})$ and $H\in C^{1,\alpha}(\overline{\Omega})$
such that $H$ minimizes the energy
$$\frac{1}{2}\int_\Omega|\nabla H|^2D(u)\,dx$$
among all $H\in W^{1,2}(\Omega)$ such that $H-h\in W^{1,2}_0(\Omega)$, or, equivalently,
$H$ is the unique weak solution to
$$\begin{cases}
\dive(D(u)\nabla H)=0&\hbox{in}~\Omega\\
H=h&\hbox{on}~\partial\Omega,
\end{cases}$$
and, in addition, $H$ is the mean curvature of the graph of $u$ with prescribed values on $\partial\Omega$:
$$\begin{cases}
\displaystyle \frac{1}{n}\dive\bigg(\frac{\nabla u}{D(u)}\bigg)=H&\hbox{in}~\Omega\\
u=g&\hbox{on}~\partial\Omega.
\end{cases}$$
\end{thm}

\begin{rem}[Nonexistence of solutions]\label{rem:nonexistence}
The conditions imposed on the curvature datum $h$ at the boundary in Theorem \ref{thm:simplifiedexistence}
come from restrictions already present when one seeks for solutions to the prescribed mean curvature equation.
Indeed, the equation for $H$ is uniformly elliptic when $u$ is, say,
Lipschitz continuous and, therefore, is always solvable. Next, if condition \eqref{eq.condh1} is not satisfied, that is,
\[ 
|h(y_0)|>\frac{n-1}{n}H_{\partial\Omega}(y_0)\qquad\hbox{for some}~y_0\in\partial\Omega
\]
and $h\geq 0$ (or $h\leq 0$) on $\partial\Omega$ then $H\geq 0$ (or $H\leq 0$) in $\Omega$
and we have that for any $\varepsilon>0$ there exists $g\in C^\infty(\overline\Omega)$ with $|g|<\varepsilon$
such that the prescribed mean curvature equation with curvature $H$ and boundary values $g$ is not solvable.
This is due to the fact that one cannot guarantee boundary gradient estimates, see \cite[Section~14.4]{Gilbarg-Trudinger}.

On the other hand, a necessary condition for existence of solutions of the prescribed mean curvature equation is
that $H$ satisfies
\begin{equation}\label{eq.condH3}
\left|\int_\Omega H\eta\,dx\right|\leq \frac{(1-\varepsilon_0)}{n}\int_\Omega|\nabla \eta|\,dx
\end{equation}
for some $\varepsilon_0>0$, for all $\eta\in C^1_c(\Omega)$,
see \cite{Giaquinta,Gilbarg-Trudinger} and Section \ref{section:weakexistence}.
It turns out that \eqref{eq.conLnH} implies \eqref{eq.condH3}.
This structural condition on $H$ can be guaranteed by imposing \eqref{eq.condh2}.
\end{rem}

\section{Existence and regularity for the geometric energy}\label{section:geometricexistence}

In this section we discuss how the technique we developed in the previous section can be applied to the geometric energy functional
$$E[\M]=\frac{1}{2}\int_\Omega\bigg[|\nabla H|^2-\bigg|\frac{\nabla u\cdot\nabla H}{D(u)}\bigg|^2\bigg]D(u)\,dx.$$
Let $\Omega$, $\alpha$, $h$ and $g$ be as in Section \ref{section:simplifiedexistence}.
Fix $v\in C^{1,\alpha}(\overline{\Omega})$ such that $v=g$ on $\partial\Omega$. Let
$$E_v[H]:=\frac{1}{2}\int_\Omega\bigg[|\nabla H|^2-\bigg|\frac{\nabla v\cdot\nabla H}{D(v)}\bigg|^2\bigg]D(v)\,dx=\int_\Omega
L(\nabla H)\,dx$$
where the smooth Lagrangian $L$ is given by
$$L(p)=\frac{1}{2}\bigg[|p|^2-\bigg|\frac{\nabla v\cdot p}{D(v)}\bigg|^2\bigg]D(v)\qquad\hbox{for}~p\in\R^n.$$
It can be seen that $L$ is coercive and uniformly convex, with
$$L_{p_ip_j}(p)\xi_i\xi_j =D(v)|\xi|^2-\frac{(\nabla v\cdot\xi)^2}{D(v)}\geq\frac{1}{D(v)}|\xi|^2.$$
Thus, there exists a unique minimizer $H\in W^{1,2}(\Omega)$ of $E_v[H]$
such that $H-h\in W^{1,2}_0(\Omega)$. In particular, $H$ is the unique weak solution to
$$\begin{cases}
\dive(a(x)\nabla H)=0&\hbox{in}~\Omega\\
H=h&\hbox{on}~\partial\Omega
\end{cases}$$
where $a^{ij}(x)=\delta_{ij}D(v)-v_{x_i}v_{x_j}/D(v)\in C^{0,\alpha}(\overline{\Omega})$ is
uniformly elliptic. Then~$H\in C^{1,\alpha}(\overline{\Omega})$~and
$$\|H\|_{C^{1,\alpha}(\overline{\Omega})}\leq C_n[\partial\Omega]_{C^{1,\alpha}}
\|v\|_{C^{1,\alpha}(\overline{\Omega})}\|h\|_{C^{1,\alpha}(\partial\Omega)}.$$ 
If $h$ satisfies \eqref{eq:h2} and \eqref{eq:h1} then Theorem \ref{thm.pmc} guarantees the existence
and uniqueness of the solution $u\in C^{2,\alpha}(\overline{\Omega})$ to \eqref{eq:pmc}. From here on we can
continue with the arguments we did in Section \ref{section:simplifiedexistence} and conclude the following result.

\begin{thm}[Existence for the geometric functional and regularity]\label{thm:geometricexistence}
Let $\Omega\subset\R^n$ be a bounded domain with $C^{2,\alpha}$ boundary $\partial\Omega$, for some $0<\alpha<1$.
Fix $g\in C^{2,\alpha}(\overline{\Omega})$. Let $h\in C^{1,\alpha}(\overline{\Omega})$ such that
$$\max_{\partial\Omega}|h|<\left(\frac{|B_1|}{|\Omega|}\right)^{1/n}$$
and
$$|h(y)|\leq\frac{n-1}{n}H_{\partial\Omega}(y)\qquad\hbox{for all}~y\in\partial\Omega,$$
where $H_{\partial\Omega}$ is the mean curvature of $\partial\Omega$ corresponding to the inner unit normal vector to $\partial\Omega$.
Then there exist $u\in C^{2,\alpha}(\overline{\Omega})$ and $H\in C^{1,\alpha}(\overline{\Omega})$
such that $H$ minimizes the energy
$$\frac{1}{2}\int_\Omega\bigg[|\nabla H|^2-\bigg|\frac{\nabla u\cdot\nabla H}{D(u)}\bigg|^2\bigg]D(u)\,dx$$
among all $H\in W^{1,2}(\Omega)$ such that $H-h\in W^{1,2}_0(\Omega)$, or, equivalently,
$H$ is the unique weak solution to
$$\begin{cases}
\dive(a(x)\nabla H)=0&\hbox{in}~\Omega\\
H=h&\hbox{on}~\partial\Omega,
\end{cases}$$
where
$$a^{ij}(x)=\delta_{ij}D(u)-\frac{u_{x_i}u_{x_j}}{D(u)}$$
and, in addition, $H$ is the mean curvature of the graph of $u$ with prescribed values on $\partial\Omega$:
$$\begin{cases}
\displaystyle \frac{1}{n}\dive\bigg(\frac{\nabla u}{D(u)}\bigg)=H&\hbox{in}~\Omega\\
u=g&\hbox{on}~\partial\Omega.
\end{cases}$$
\end{thm}

\section{Existence and regularity of variational solutions}\label{section:weakexistence}

We next develop the variational formulation to solve \eqref{eq.min}.
It is important to notice that the main result of this section, Theorem \ref{thm.weak}, is of a different nature than Theorems \ref{thm:simplifiedexistence}
and \ref{thm:geometricexistence}. Indeed, in Theorem \ref{thm.weak} we construct a minimizing pair $(u,H)$.

We start by recalling (see \cite{Giaquinta}) that $u\in \BV(\Omega)$ (the space of functions of bounded variation in $\Omega$)
is a generalized solution to the prescribed mean curvature equation
with (weak) mean curvature $H\in L^1(\Omega)$ and boundary value $g\in L^1(\partial\Omega)$ if 
\begin{equation}\label{eq.weakpmc}\tag{WPMC}
\J[u]=\min_{v\in\mathrm{BV}(\Omega)}\J[v]
\end{equation}
where
\[
\J[v]:=\int_\Omega D(v)+\int_\Omega nHv\,dx+\int_{\partial\Omega}|v-g|\,dS \\
\]
and
$$\int_\Omega D(v):=\sup\bigg\{\int_\Omega\bigg[v\sum_{i=1}^n\partial_{x_i}\phi_i+\phi_{n+1}\bigg]\,dx:
\phi_i\in C^1_c(\Omega),~\sum_{i=1}^{n+1}\phi_i^2\leq1\bigg\}.$$
It can be seen that (see \cite{Giusti2}) for $v\in W^{1,1}(\Omega)$ we have
\[
\int_\Omega D(v)=\int_\Omega(1+|\nabla v|^2)^{1/2}\,dx.
\]

In \cite{Giaquinta}, Giaquinta proved that if $H$ is a measurable function then
\eqref{eq.weakpmc} is solvable in $\BV(\Omega)$ if and only if there is $\varepsilon_0>0$ such that, for any
set of finite perimeter $A\subset\Omega$,
\begin{equation}\label{eq:Giaquintacondition}
\bigg|\int_AH\,dx\bigg|\leq(1-\varepsilon_0)\frac{1}{n}P(\partial A)
\end{equation}
where $P(\partial A)$ denotes the perimeter of $A$.

The area measure is defined by 
$$D(u)(U)=\sup\bigg\{\int_\Omega\bigg[u\sum_{i=1}^n\partial_{x_i}\phi_i+\phi_{n+1}\bigg]\,dx:
\phi_i\in C^1_c(U),~\sum_{i=1}^{n+1}\phi_i^2\leq1\bigg\}$$
for any $U\subset\subset\Omega$ open and
$D(u)(V)=\inf\left\{D(u)(U):V\subset U\text{ and }U\text{ is open}\right\}$,
whenever $V\subset\Omega$ is arbitrary.
It can be seen that if $u\in\BV(\Omega)$ then $D(u)$ is a Radon measure on $\R^n$.

From now on, we fix a bounded $C^2$ domain $\Omega$ and $g\in C^{1,\alpha}(\partial\Omega)$
for some $0<\alpha<1$. We consider the minimization problem
\[
\min_{(u,H)\in\A}\I[u,H]
\]
where
\begin{equation}\label{eq.I}
\I[u,H]:=\int_\Omega|\nabla H|^2\,dD(u)
\end{equation}
and $dD(u)$ stands for the area measure defined above. 
The admissible set $\A$ is defined as follows. Let $h\in W^{2,2}(\Omega)\cap\mathrm{Lip}(\partial{\Omega})$ satisfying 
\begin{equation}\label{eq.conh}
|h(y)|\leq\frac{n-1}{n}H_{\partial\Omega}(y),~y\in\partial\Omega,
\qquad\hbox{and}\qquad\max_{\partial\Omega}|h|\leq(1-\varepsilon_0)\left(\frac{|B_1|}{|\Omega|}\right)^{1/n},
\end{equation}
where $H_{\partial\Omega}(y)$ is the mean curvature of $\partial\Omega$ at $y\in\partial\Omega$, for some $0<\varepsilon_0<1$.
Define, for some $C_0>0$,
\begin{equation}\label{eq.A}
\A:=\left\{
\begin{array}{c}
(u,H)\in\BV(\Omega)\times(\Lip(\overline{\Omega})\cap W^{2,2}(\Omega)):u\text{ solves }\eqref{eq.weakpmc} \\
\hbox{and }\|H\|_{\Lip(\overline{\Omega})}+\|H\|_{W^{2,2}(\Omega)}\leq C_0,~H=h\text{ on }\partial\Omega
\end{array}
\right\}.
\end{equation}

\begin{thm}[Existence and regularity of variational solutions]\label{thm.weak}
Let $\Omega$ be a bounded domain with $C^2$ boundary $\partial\Omega$. Let $g\in C^{1,\alpha}(\partial\Omega)$
for some $0<\alpha<1$, and $h\in W^{2,2}(\Omega)\cap \mathrm{Lip}(\partial{\Omega})$
satisfying \eqref{eq.conh} for some $0<\varepsilon_0<1$. Let $\I$ be defined by \eqref{eq.I}.
Then there is $C_0>0$, depending only on $\partial\Omega$ and $\|h\|_{L^\infty(\partial\Omega)}$,
such that the admissible set $\A$ in \eqref{eq.A} is nonempty
and there exists a minimizer $(u_\infty,H_\infty)$ of $\I$ within the class $\A$.
Moreover, $u_\infty\in C^{1,\alpha}(\overline{\Omega})\cap C^{2,\alpha}_{\mathrm{loc}}(\Omega)$.
\end{thm}

To prove Theorem \ref{thm.weak} we need to recall the notion and properties of $\Gamma-$convergence (see \cite{Braides}) in our context.
Let $\J_k$, $k\geq1$, and $\J_\infty$ be functionals defined on $\BV(\Omega)$ and taking values in $[-\infty,\infty]$.
Then $\{\J_k\}_{k\geq1}$ is said to $\Gamma-$converge to $\J_\infty$ if the following two conditions hold:
\begin{enumerate}[$(a)$]
\item For every $v\in \BV(\Omega)$ and every sequence $\{v_k\}_{k\geq1}\subset\BV(\Omega)$
such that $v_k\rightarrow v$ in $\BV(\Omega)$ it holds 
\[
\liminf_{k\rightarrow\infty}\J_k(v_k)\geq \J_\infty(v).
\]
\item For every $v\in \BV(\Omega)$ there exists a sequence $\{v_k\}_{k\geq1}\subset\BV(\Omega)$ such that $v_k\rightarrow v$ in $\BV(\Omega)$ for which
\[
\limsup_{k\rightarrow\infty}\J_k(v_k)\leq \J_\infty(v).
\]
\end{enumerate}
We will use the following fact (see \cite[Theorem~1.21]{Braides}).
Suppose that $\{\J_k\}_{k\geq1}$ is an equi-mildly coercive
sequence of functionals on $\BV(\Omega)$ that $\Gamma-$converges to $\J_\infty$. Then there exits
\[
\min_{\BV(\Omega)}\J_\infty=\lim_{k\rightarrow\infty}\inf_{\BV(\Omega)}\J_k.
\]
Moreover, if $\{u_k\}_{k\geq1}\subset\BV(\Omega)$ is a precompact sequence in $\BV(\Omega)$ such that 
\[
\lim_{k\rightarrow\infty}\J_k(u_k)=\lim_{k\rightarrow\infty}\inf_{\BV(\Omega)}\J_k
\]
then every limit of $\{u_k\}_{k\geq1}$ is a minimum point for $\J_\infty$.
A functional $\J$ is mildly coercive in $\BV(\Omega)$ if there exists a nonempty compact set $K\subset\BV(\Omega)$
such that $\inf_K\J=\inf_{\BV(\Omega)}\J$, and equi-mild coercivity means that the set $K$
is the same for the whole sequence $\{\J_k\}_{k\geq1}$.

\begin{proof}[Proof of Theorem \ref{thm.weak}]
We begin by showing that $\A$ is nonempty. Let $H$ be the harmonic function in $\Omega$ such that $H=h$ on $\partial\Omega$.
By elliptic regularity, $H\in\Lip(\overline{\Omega})\cap W^{2,2}(\Omega)$ and
$\|H\|_{\Lip(\overline{\Omega})}+\|H\|_{W^{2,2}(\Omega)}\leq C_0$,
where $C_0=C_0(\partial\Omega,\|h\|_{L^\infty(\partial\Omega)})>0$. 
Moreover, by the H\"older and isoperimetric inequalities, 
\[
\bigg|\int_AH\,dx\bigg|\leq\|H\|_{L^n(\Omega)}|A|^{\frac{n-1}{n}}\leq \|H\|_{L^n(\Omega)}\frac{P(\partial A)}{n|B_1|^{1/n}}.
\]
The maximum principle and \eqref{eq.conh} give
\[
\|H\|_{L^n(\Omega)}\leq |\Omega|^{1/n}\max_{\partial\Omega}|h|\leq (1-\varepsilon_0)|B_1|^{1/n}.
\]
Therefore,
\[
\bigg|\int_AH\,dx\bigg|\leq\frac{(1-\varepsilon_0)}{n}P(\partial A)
\]
and \eqref{eq.weakpmc} is solvable for this $H$. If we let $u\in\BV(\Omega)$ be the corresponding minimizer of $\J$
then the pair $(u,H)$ is in $\A$.

We further point out that $\int_\Omega D(u)<\infty$ and $H\in\mathrm{Lip}(\overline{\Omega})$ so that 
\[
0\leq m:= \inf_{(u,H)\in\A}\I[u,H]\leq \|\nabla H\|_{L^\infty(\Omega)}^2D(u)(\Omega)<\infty.
\]

Consider next a minimizing sequence $\{(u_k,H_k)\}_{k\geq1}\subset\A$, that is,
$\lim_{k\rightarrow\infty}\I[u_k,H_k]=m$.
Since every $u_k$ is a minimizer of the functional $\J_k$ defined by
\begin{equation}\label{eq.Jk}
\J_k[v]:= \int_\Omega D(v)+\int_\Omega nH_kv\,dx+\int_{\partial\Omega}|v-g|\,dS
\end{equation}
we have that, for any $u_0\in\BV(\Omega)$, $\J_k(u_k)\leq\J_k(u_0)$, from where
\begin{equation}\label{estimatefornablauk}
\int_\Omega D(u_k)+\int_\Omega nH_ku_k\,dx\leq C+\int_\Omega nH_ku_0\,dx
\end{equation}
for $C>0$ independent of $k$. 
Now we estimate by below the second integral in the left hand side above as in \cite{Giaquinta}.
Indeed, we extend $H_k$ and $u_k$ as $0$ outside of $\Omega$ and write
$$\int_\Omega nH_ku_k\,dx=\int_{\R^n}nH_ku_k^+\,dx-\int_{\R^n}nH_ku_k^-\,dx$$
where $u^+,u^-\geq0$ denote the positive and negative parts of a function $u$, respectively. By
\eqref{eq:Giaquintacondition} and the coarea formula for functions of bounded variation (see \cite{Evans-Gariepy}),
\begin{align*}
\bigg|\int_{\R^n}nH_ku_k^{\pm}\,dx\bigg| &= \bigg|\int_0^\infty\int_{\{x:u_k^{\pm}(x)>t\}}nH_k\,dx\,dt\bigg| \\
&\leq (1-\vep_0)\int_0^\infty P(\partial\{x:u_k^{\pm}(x)>t\})\,dt \\
&=(1-\vep_0)\int_{\R^n}|\nabla u_k^{\pm}|.
\end{align*}
Hence,
\begin{align*}
\int_\Omega nH_ku_k\,dx &\geq-(1-\vep_0)\int_{\R^n}|\nabla u_k| \\
&=-(1-\vep_0)\int_{\Omega}|\nabla u_k|-(1-\vep_0)\int_{\partial\Omega}|u_k|\,dS \\
&=-(1-\vep_0)\int_{\Omega}|\nabla u_k|-C
\end{align*}
as $u_k=g$ on $\partial\Omega$ for all $k$ (see \cite{Giaquinta}). Using this in \eqref{estimatefornablauk},
$$\int_\Omega D(u_k) \leq (1- \varepsilon_0)\int_\Omega|\nabla u_k| + n\|H_k\|_{L^\infty(\Omega)}\|u_0\|_{L^1(\Omega)}+C$$
for a new constant $C>0$ that is independent of $k$. Moreover, the uniform bound on the
$L^\infty(\Omega)$ norm of $\{H_k\}_{k\geq1}$ gives
\[
\int_\Omega|\nabla u_k| \leq (1- \varepsilon_0)\int_\Omega|\nabla u_k| + nC_0\|u_0\|_{L^1(\Omega)}+C.
\]
so that 
\[
\vep_0\int_\Omega|\nabla u_k| \leq nC_0\|u_0\|_{L^1(\Omega)}+C.
\]
Hence, by compactness in $\BV(\Omega)$, there exist a subsequence of $\{u_k\}_{k\geq1}$,
still denoted by the same indexes, and $u_\infty\in \BV(\Omega)$ such that 
$u_k\rightarrow u_\infty$ in $L^1(\Omega)$ as $k\rightarrow\infty$, and
$|\nabla u_\infty|(U)\leq \liminf_{k\to\infty}|\nabla u_k|(U)$ for any Borel set $U\subset\Omega$. In addition,
\begin{equation}\label{eq.lscmeasures}
D(u_\infty)(U)\leq \liminf_{k\rightarrow\infty} D(u_k)(U).
\end{equation}
By Poincar\'e's inequality and the Rellich--Kondrachov compactness theorem,
there exist a subsequence of $\{H_k\}_{k\geq1}$, still denoted by the same indexes, and $H_\infty\in W^{2,2}(\Omega)$ such that
\begin{equation}\label{eq.nablaH}
\nabla H_k\rightarrow \nabla H_\infty\qquad\hbox{in}~L^2(\Omega),~\hbox{as}~k\rightarrow\infty.
\end{equation}
Further, due to the uniform bound on $\|H_k\|_{\Lip(\overline{\Omega})}$,
we may assume that $H_k$ and $\nabla H_k$ converge weak-$\ast$ in $L^\infty(\Omega)$ to $H_\infty\in\Lip(\overline{\Omega})$.
Finally, this and the weak convergence of $H_k$ to $H_\infty$ in $W^{2,2}(\Omega)$ ensure that
\[
\|H_\infty\|_{\Lip(\overline{\Omega})}+\|H_\infty\|_{W^{2,2}(\Omega)}\leq C_0.
\]

Let us now prove that $(u_\infty,H_\infty)\in\A$. Recall the functionals $\J_k$ defined in \eqref{eq.Jk}
for the subsequence $H_k$ we just found, and define $\J_\infty$ analogously.
We claim that $\{\J_k\}_{k\geq1}$ $\Gamma-$converges to $\J_\infty$.
Indeed, it is sufficient to prove the $\Gamma-$convergence of 
\[
\widetilde{\J}_k(v):=\int_\Omega vH_k\,dx
\qquad\hbox{to}\qquad
\widetilde{\J}_\infty(v):=\int_\Omega vH_\infty\,dx
\]
since the other two terms are continuous perturbations of $\J_k$ (see \cite{Braides}).
To prove the liminf inequality $(a)$, let $\{v_k\}_{k\geq1}\subset\BV(\Omega)$ and $v\in\BV(\Omega)$
such that $v_k\rightarrow v$ in $\BV(\Omega)$. We write
\[
\int_\Omega v_kH_k\,dx-\int_\Omega vH_\infty\,dx = I_k+II_k+III_k
\]
with
\begin{align*}
I_k&=\int_\Omega (v_k-v)H_\infty\,dx\\
II_k&=\int_\Omega (v_k-v)(H_k-H_\infty)\,dx\\
III_k&=\int_\Omega v(H_k-H_\infty)\,dx.
\end{align*}
By lower semicontinuity (as in \cite[Proposition 2.1]{Giusti1}), $\liminf_{k\rightarrow\infty}I_k\geq0$. Next, we bound
\[
|II_k|\leq \|v_k-v\|_{L^1(\Omega)}\left(\|H_k\|_{L^\infty(\Omega)}+\|H_\infty\|_{L^\infty(\Omega)}\right).
\] 
The convergence of $v_k$ to $v$ in $L^1(\Omega)$ and the uniform bound of $H_k$ in $L^\infty(\Omega)$ give $\lim_{k\rightarrow\infty}II_k=0$. 
Finally, $\lim_{k\rightarrow\infty}III_k=0$ by the weak-$\ast$ convergence of
$H_k$ to $H_\infty$ in $L^\infty(\Omega)$.
As for the limsup inequality $(b)$, given any $v\in\BV(\Omega)$, consider the constant sequence $v_k=v$ for all $k\geq1$ and notice that,
using the weak-$\ast$ convergence of $H_k$ to $H_\infty$ in $L^\infty(\Omega)$,
\[
\lim_{k\rightarrow\infty}\widetilde{\J}_k(v_k)=\widetilde{\J}_\infty(v).
\]
Hence, $\{\J_k\}_{k\geq1}$ converges to $\J_\infty$ in the $\Gamma$ sense.
Furthermore, the sequence $\{\J_k\}_{k\geq1}$ is equi-mildly coercive. Indeed,
by regularity estimates for the prescribed mean curvature equation, all the minimizers
of each $\J_k$, $k\geq1$, are in a ball in $C^{1,\alpha}(\overline{\Omega})$ of radius depending only on
$C_0$, $\Omega$, $g$ and $h$, and this ball is compact in $\BV(\Omega)$.

Consequently,
$u_\infty$ is a minimizer of  $\J_\infty$ and $(u_\infty,H_\infty)\in\mathcal{A}$.

Finally, by convexity,
\begin{align*}
\int_\Omega|\nabla H_k|^2\,dD(u_k)&\geq \int_\Omega|\nabla H_\infty|^2\,dD(u_k)\\
&\quad+2\int_\Omega\nabla H_\infty\cdot(\nabla H_k-\nabla H_\infty)\,dD(u_k).
\end{align*}
As $k\to\infty$, the left hand side converges to $m$. As for the right hand side, \eqref{eq.lscmeasures} implies that 
$$\liminf_{k\to\infty}\int_\Omega|\nabla H_\infty|^2\,dD(u_k)
\geq\int_\Omega|\nabla H_\infty|^2\,dD(u_\infty).$$
For the second term on the right hand side above, the regularity of the prescribed mean
curvature equation gives that $u_k\in\Lip(\overline{\Omega})$, with Lipschitz constant bounded by
some constant $c>0$ independent of $k$
(see \cite{Giaquinta}). Hence,
$$\bigg|\int_\Omega\nabla H_\infty\cdot(\nabla H_k-\nabla H_\infty)\,dD(u_k)\bigg|
\leq (1+c^2)^{1/2}\|\nabla H_{\infty}\|_{L^2(\Omega)}\|\nabla H_k-\nabla H_\infty\|_{L^2(\Omega)}.$$
In view of \eqref{eq.nablaH}, this term goes to 0 as $k\to\infty$. We have shown that $(u_\infty,H_\infty)$ is a minimizer. 

Since the first relation in \eqref{eq.conh} holds, by the results in \cite{Giaquinta} (for which $\partial\Omega\in C^2$ is in fact enough, see also \cite[Theorem 13.2]{Gilbarg-Trudinger} and the comment following it), we conclude that $u\in C^{1,\alpha}(\overline{\Omega})\cap C^{2,\alpha}_{\mathrm{loc}}(\Omega)$.
\end{proof}

\begin{rem}
We point out that the first condition in \eqref{eq.conh} is only used to deduce the boundary regularity of $u_\infty$,
and is not actually needed in the construction of the pair $(u_\infty,H_\infty)$.
\end{rem}

\medskip

\noindent\textbf{Acknowledgements.~}We would like to thank Irene Mart\'inez Gamba and the referee
for useful comments that helped improve the presentation of the paper.

\section*{Statements and declarations}

\noindent\textbf{Data Availability Statement.~}Data sharing not applicable to this article as no datasets were generated or analyzed during the current study.

\medskip

\noindent\textbf{Conflict of Interest Statement.~}The authors declare no conflict of interest.



\end{document}